\documentclass[11pt]{amsart}
\usepackage[english]{babel}
\usepackage[T1]{fontenc}
\usepackage[ansinew]{inputenc}
\usepackage{amsmath}
\usepackage{amsfonts}
\usepackage{ mathrsfs }
\usepackage{amssymb}
\usepackage{textcomp}
\usepackage{enumitem}
\usepackage[hidelinks]{hyperref}
\usepackage[arrow, matrix, curve]{xy}
\usepackage{comment}
\usepackage{color}
\usepackage{tikz}
\numberwithin{paragraph}{section}
\numberwithin{equation}{section}

\newtheorem{satz}{Theorem}[section]

\newtheorem{lem}[satz]{Lemma}

\newtheorem{prop}[satz]{Proposition}

\newtheorem{kor}[satz]{Corollary}
\theoremstyle{definition}
\newtheorem{defn}[satz]{Definition}
\newtheorem{bem}[satz]{Remark}
\newtheorem{Ex}[satz]{Example}
\newtheorem{Const}[satz]{Construction}

\newtheorem*{ack}{Acknowledgements}

\newtheorem{theointro}{Theorem}

\newcommand{\Z}{\mathbb{Z}}

\newcommand{\R}{\mathbb{R}}
\newcommand{\C}{\mathbb{C}}
\newcommand{\T}{\mathbb{T}}

\newcommand{\G}{\mathbb{G}}

\newcommand{\Xcal}{\mathcal{X}}

\newcommand{\Ocal}{\mathcal{O}}

\newcommand{\Linear}{\mathbb{L}}

\newcommand{\CC}{\mathcal{C}}

\newcommand{\Xan}{X^{\an}}

\newcommand{\abs}{\vert \,.\, \vert}
\newcommand{\Htilde}{\widetilde{\mathscr{H}(x)}}
\newcommand{\Ktilde}{\tilde{K}}
	\DeclareMathOperator{\an}{an}

	\DeclareMathOperator{\trop}{trop}

	\DeclareMathOperator{\Hom}{Hom}
	
	\DeclareMathOperator{\Spec}{Spec}

	\DeclareMathOperator{\Trop}{Trop}

	\DeclareMathOperator{\supp}{supp}

	\DeclareMathOperator{\divisor}{div}
	
	\DeclareMathOperator{\Pic}{Pic}

	\DeclareMathOperator{\inn}{in}
	\DeclareMathOperator{\red}{red}

	\DeclareMathOperator{\Monoids}{Monoids}
	\DeclareMathOperator{\val}{val}
	
	\DeclareMathOperator{\HH}{H}
	\DeclareMathOperator{\Div}{Div}
	\DeclareMathOperator{\Hscr}{\mathscr{H}}	
	\DeclareMathOperator{\Prin}{Prin}
	\DeclareMathOperator{\trdeg}{trdeg}
	\DeclareMathOperator{\fin}{fin}

\newcommand{\Pbb}{\mathbb{P}}

\DeclareMathOperator{\XS}{\mathcal{X}}
\DeclareMathOperator{\YS}{\mathcal{Y}}

\DeclareMathOperator{\A}{\mathbb{A}}

\def\quotient#1#2{\raise0.75ex\hbox{$\,#1$}\big/\lower0.75ex\hbox{$#2\,$}}

\title[Smooth and fully faithful tropicalizations for Mumford curves]{Constructing smooth and fully faithful tropicalizations for Mumford curves}

\author[P.~Jell]{Philipp Jell}

\address{Philipp Jell, 
Universit\"at Regensburg,
93040 Regensburg}
\email{philipp.jell@ur.de}

\thanks{The author was supported by the DFG Research Fellowship  JE 856/1-1 and 
by the Institute Mittag-Leffler and the ``Vergstiftelsen''.}

\begin{document}
\begin{abstract}
The tropicalization of an algebraic variety $X$ is a combinatorial shadow of $X$, 
which is sensitive to a closed embedding of $X$ into a toric variety. 
Given a good embedding, the tropicalization can provide a lot of information about $X$.
We construct two types of these good embeddings for Mumford curves:  
fully faithful tropicalizations, which are embeddings such that the tropicalization admits 
a continuous section to the associated Berkovich space $\Xan$ of $X$, and 
smooth tropicalizations. 
We also show that a smooth curve that admits a smooth tropicalization is 
necessarily a Mumford curve. 
Our key tool is a variant of a lifting theorem for rational functions on metric graphs.

\bigskip

\noindent
MSC: Primary: 14T05; Secondary: 14G22, 32P05

\bigskip

\noindent
Keywords: Tropical geometry, smooth tropical curves, Mumford curves, extended skeleta, 
faithful tropicalization
\end{abstract}

\maketitle 
%\tableofcontents

\section{Introduction}

Let $K$ be a field that is algebraically closed and complete 
with respect to a non-archimedean non-trivial absolute value. 
Given a closed subvariety $X$ of a toric variety $Y$ over $K$, 
one can associate a so-called tropical variety $\Trop(X)$ 
which is a polyhedral complex.
Note however, that $\Trop(X)$ is not an invariant of $X$, but depends on the embedding into $Y$. 

In good situations, $\Trop(X)$ can retain a lot of information about $X$.
Let us mention here work by Katz, Markwig and Markwig 
on the $j$-invariant of elliptic curves \cite{KMM1, KMM2} and 
work by Itenberg, Mikhalkin, Katzarkov and Zharkov 
on recovering Hodge numbers in degenerations of complex projective 
varieties \cite{IKMZ}.

In the latter work, a smoothness condition for tropical varieties in arbitrary codimension appears: 
a tropical variety is called \emph{smooth} if it is locally isomorphic to the Bergman fan of a matroid. 
(See Definition \ref{defn smooth} for an equivalent definition for curves.)
For tropical hypersurfaces, this is equivalent to the associated subdivision of the Newton polytope 
being a primitive triangulation, which is the definition of smoothness that is generally used for tropical hypersurfaces
\cite[Remark p. 24]{IKMZ}. 

The definition in \cite{IKMZ} is motivated by complex analytic geometry. 
A complex variety is smooth if it is locally isomorphic to open subsets of $\C^n$ in the analytic topology. 
Bergman fans of matroids are the local models for linear spaces in tropical geometry, 
thus it makes sense to call a tropical variety smooth if it is locally isomorphic to the Bergman fan of a matroid.

This smoothness condition has been shown to imply many tropical analogues of classical theorems
from complex and algebraic geometry, for example intersection theory, 
Poincar\'e duality and a Lefschetz $(1,1)$-theorem \cite{Shaw:IntMat, JSS, JRS}.

In this paper, we investigate the question for which smooth projective curves 
there exist closed embeddings $\varphi$ into toric varieties
such that $\Trop_{\varphi}(X) := \Trop(\varphi(X))$ is smooth. 
The answer turns out to be Mumford curves (see Definition \ref{defn Mumford}). 
Indeed, we show that for these curves 
we can ``repair'' any given embedding by passing to a refinement 
(see Definition \ref{defn refinement} for a definition of refinement).

\begin{theointro} [Theorem \ref{main theorem}, Theorem \ref{main theorem II}] \label{thmA}
Let $X$ be a smooth projective curve of positive genus. 
Then the following are equivalent:
\begin{enumerate}
\item
$X$ is a Mumford curve. 
\item
There exists a closed embedding $\varphi \colon X \to Y$ for a toric variety $Y$ that meets the dense torus
such that $\Trop(\varphi(X))$ is a smooth tropical curve.
\item
Given a closed embedding $\varphi \colon X \to Y$ of $X$ into a toric variety $Y$ that meets the dense torus,
there exists a refinement $\varphi' \colon X \to Y'$ of $\varphi$  
such that $\Trop(\varphi'(X))$ is a smooth tropical curve. 
\end{enumerate}
\end{theointro}

Denote by $\Xan$ the Berkovich analytification of $X$ \cite{BerkovichSpectral}.
We give alternative characterizations of Mumford curves in terms of $\Xan$ in Remark \ref{bem Mumford}. 
Theorem \ref{thmA}, specifically the equivalence of i) and ii), 
may be viewed as an alternative characterization that is purely tropical.

Payne showed in \cite[Theorem 4.2]{Payne} that we have a homeomorphism 
\begin{align}
\Xan = \varprojlim_{\varphi \colon X \to Y} \Trop_{\varphi}(X).
\end{align}
Theorem \ref{thmA} shows that if $X$ is a Mumford curve
we can let the limit on the right hand side as well run only over closed embeddings $\varphi$ 
such that $\Trop_{\varphi}(X)$ is a smooth tropical curve, 
meaning the smoothness on the left hand side is reflected on the right hand side. 

Another often used property of tropicalizations is faithfulness. 
For curves this means that given a finite skeleton $\Gamma$ of $\Xan$, one requires that 
$\varphi_{\trop} := \trop \circ \varphi^{\an}$ is a homeomorphism from $\Gamma$ onto its image, 
preserving the piecewise linear structure. 
Existence of faithful tropicalizations was proved by Baker, Payne and Rabinoff for curves 
and generalized to higher dimension by Gubler, Rabinoff and Werner \cite{BPR, GRW}. 
For further work on faithful tropicalizations see for example \cite{CuetoMarkwig, Manjunath, KY, Wagner}.

Baker, Payne and Rabinoff also introduced so-called completed extended skeleta for curves. 
For a smooth projective curve $X$, these are metric subgraphs $\Sigma$ of $\Xan$, 
potentially with edges of infinite length, 
that come with a canonical retraction $\tau \colon \Xan \to \Sigma$. 
Given a closed embedding $\varphi \colon X \to Y$ for $Y$ a toric variety with dense torus $T$, there exists 
an \emph{associated complete skeleton}
$\Sigma(\varphi)$, which has the property that 
$\varphi_{\trop}$ factors through the retraction $\tau \colon \Xan \to \Sigma(\varphi)$ 
(see Definition \ref{defn minimal skeleton}).
Denote by $X^\circ := \varphi^{-1}(T)$. 
We call $\varphi_{\trop}$ \emph{fully faithful} if 
$\varphi_{\trop}$ maps $\Sigma(\varphi)$ homeomorphically onto its image and is an 
isometry when restricted to $\Sigma(\varphi) \cap X^{\circ, \an}$.
Note that this is much stronger than a faithful tropicalization, 
since by definition the image of $\Sigma(\varphi)$ is $\Trop_{\varphi}(X)$.

We prove the following fully faithful tropicalization result. 

\begin{theointro} [Theorem \ref{prop edge}] \label{thmB}
Let $X$ be a Mumford curve and $\varphi \colon X \to Y$ a closed embedding into 
a toric variety $Y$ that meets the dense torus. 
Then there exists a refinement $\varphi'$ of $\varphi$
that is fully faithful. 
\end{theointro}

As a direct consequence of the fact that $\varphi'$ is fully faithful, 
we obtain a continuous section 
$s \colon \Trop_{\varphi'}(X) \to \Xan$ of $\varphi'_{\trop}$ by 
composing the inverse of $\varphi'_{\trop}|_{\Sigma(\varphi')}$ with the inclusion of $\Sigma(\varphi')$ into $\Xan$ 
(see Corollary \ref{cor section}). 
Such sections, though only defined on subsets of $\Trop_{\varphi}(X)$, 
were also constructed in \cite[Theorem 5.24]{BPR} and \cite[Theorem 8.15]{GRW2}.

For reader interested in effective bounds on the dimensions of the ambient toric varieties, let us mention \cite{GunnJell}, 
where Gunn and the author construct fully faithful tropicalizations in ambient dimension 3, 
and also give bounds on the ambient dimensions for smooth tropicalizations. 

We prove Theorem \ref{thmB} as a first step to prove Theorem \ref{thmA}, 
more precisely that i) implies iii) therein. 
Our techniques to prove these results  
are based on 
the following lifting theorem for rational functions on metric graphs, 
which is a variant of a theorem by Baker and Rabinoff \cite[Theorem 1.1]{BRab}.
The relevant notions are recalled in Section \ref{functions and divisors}. 

\begin{theointro} [Theorem \ref{lifting theorem}] \label{thmC}
Let $X$ be a Mumford curve and $\Gamma$ be a finite skeleton with retraction $\tau$. 
Let $D \in \Div(X)$ be a divisor of degree $g$ and 
let $B = p_1 + \dots + p_g \in \Div(\Gamma)$ be a break divisor 
such that $\tau_* D - B$ is a principal divisor on $\Gamma$. 
Assume that $B$ is supported on two-valent points of $\Gamma$. 
Then there exist $x_i \in X(K)$ such that $\tau_* x_i = p_i$ and such that 
$D - \sum_{i=1}^g  x_i$ is a principal divisor on $X$. 
\end{theointro}

Theorem \ref{thmC} is of independent interest, 
since, given a skeleton of $X$, it enables one to construct closed embeddings with nice tropicalizations.
We treat an example of this in Example \ref{tate curves} for a genus $1$ Mumford curve (also called a Tate curve). 

We give an idea of the proof of Theorem \ref{thmB}, which is carried out in Section \ref{section ff trop}.
Given an edge $e$ of $\Sigma(\varphi)$, using Theorem \ref{thmC}, 
we construct a rational function $f_e \in K(X)^*$ in such a way 
that $\log \vert f \vert$ has slope $1$ along $e$.
Considering the embedding $\varphi' := (\varphi, f_e) \colon X \to Y \times \Pbb^1$, 
this ensures that $\varphi'_{\trop}$ maps $e$ homeomorphically onto its image and 
that the corresponding stretching factor equals $1$ (see Definition \ref{weight} for the definition of stretching factor). 
Using a good choice of $D \in \Div(X)$ and $B \in \Div(\Gamma)$, 
Theorem \ref{thmC} moreover allows us to construct $f_e$ in such a way that 
the same holds for all edges of $\Sigma(\varphi')$ that 
are not contained in $\Sigma(\varphi)$. 
Doing so for every edge of $\Sigma(\varphi)$, we obtain Theorem \ref{thmB}. 

In Section \ref{section smooth trop}, we proceed similarly for smoothness
and thus prove that i) implies iii) in Theorem \ref{thmA}.

In Section \ref{section smooth mumford} we prove that, for a smooth projective curve $X$, 
the existence of a closed embedding with a smooth tropicalization 
already implies that $X$ is a Mumford curve. 
The key result we use is a joint observation by Mikhalkin, Sturmfels and Ziegler \cite{MikOberwolfach}, 
which states that a variety whose tropicalization is a tropical linear space 
is actually a linear space (see Theorem \ref{duck theorem}).
The version of the theorem we use was proved by Katz and Payne \cite{KatzPayne}
and works for trivially valued fields in any characteristic (see Theorem \ref{duck theorem}).
We also show that if $\Trop_\varphi(X)$ is smooth then $\varphi_{\trop}$ is necessarily fully faithful 
(see Theorem \ref{smooth implies ff}).

\begin{ack}
The author was inspired to reconsider the questions in this paper 
by a question asked by Hannah Markwig during an open problem session at the program
``Tropical geometry, amoebas and polytopes'' at the Institute Mittag-Leffler. 
He would like to thank Hannah Markwig for the encouragement and the Institute Mittag-Leffler 
for the wonderful working conditions. 
He would also like to thank Matt Baker, Walter Gubler, Yoav Len, Hannah Markwig, Sam Payne, 
Joe Rabinoff, Veronika Wanner and 
Annette Werner for helpful discussions and comments. 
He would also like to thank the anonymous referees for their precise reports and detailed comments.
\end{ack}

\subsection*{Conventions}

Throughout, $K$ will be an algebraically closed field that 
is complete with respect to a non-archimedean non-trivial absolute value $\abs_K$. 
We denote the value group by $\Lambda := \log \vert K^\times \vert_K$,
the valuation ring by $K^\circ$ and 
the residue field by $\Ktilde$.  
A variety over $K$ is a separated reduced irreducible scheme of finite type and 
a curve is a one-dimensional variety. 
$X$ will be a smooth projective curve over $K$. 
We will denote finite skeleta of $X$ by $\Gamma$ and 
completed extended skeleta in the sense of \cite{BPR2} by $\Sigma$. 
We will generally denote toric varieties by $Y$ and their dense tori by $T$.

%--------------------------------------------------------------------------------------------------------------------------------------

\section{Preliminaries}

\subsection{Tropical toric varieties and tropical curves}

Let $N$ be a free abelian group of rank $n$, $M := \Hom_\Z(N, \Z)$ its dual, 
$N_\R := N \otimes \R$ and $\Delta$ a rational 
pointed fan in $N_\R$. 
We write $\T := \R \cup \{-\infty\}$.

For $\sigma \in \Delta$ we define the monoid 
$S_\sigma := \{ \varphi \in M \mid \varphi(v) \geq 0 \text{ for all } v \in \sigma \}$
and write $N(\sigma) := N_\R / \langle \sigma \rangle_\R$,
where $\langle \sigma \rangle_{\R}$ denotes the real vector space spanned by $\sigma$. 
We write
\begin{align*}
N_\Delta =  \coprod \limits_{\sigma \in \Delta} N(\sigma).
\end{align*}
We endow $N_\Delta$ with a topology in the following way:

For $\sigma \in \Delta$ write $N_\sigma = \coprod \limits_{\tau \prec \sigma} N(\tau)$. 
This is naturally identified with $\Hom_{\Monoids}(S_\sigma, \T)$. 
We give $N_\sigma$ the subspace topology of $\T^{S_\sigma}$. 
For $\tau \prec \sigma$, the space $\Hom(S_\tau, \T)$ is naturally identified with the open
subspace of $\Hom_{\Monoids}(S_\sigma, \T)$ of maps that map $\tau^{\perp} \cap M$ to $\R$. 
We define the topology of $N_\Delta$ to be the one obtained by gluing along these identifications.

\begin{defn}
We call the space $N_\Delta$ a  \emph{tropical toric variety}.
\end{defn}

The space $N_\Delta$ is sometimes called the canonical compactification of $N_\R$ 
with respect to $\Delta$. 
Note that $N_\Delta$ contains $N_\R$ as a dense open subset. 

\begin{Ex} \label{example tropical toric}
Let $N = \Z^n$ with basis $x_1,\dots,x_n$ and $\Delta$ be the complete fan 
whose rays are spanned by $-x_1,\dots,-x_n$ and  $x_0 := \sum x_i$. 
For any $d$-rays there is a face $\sigma$ 
of dimension $d$ that contains exactly these rays. 
Then $N(\sigma)$ is an $n{-}d$-dimensional vector space. 
The topology is such that $N_\Delta$ is homeomorphic to an $n$-simplex, 
where $N(\sigma)$ is identified with the relative interior of a $n{-}d$-dimensional simplex in the boundary. 
For example, $N_\R$ corresponds to the vertex at the origin in $\Delta$ 
and forms the interior of $N_\Delta$ when we view $N_\Delta$ as a simplex. 

However, we will heavily use the structure of $N_\R$ as a vector space, 
so we generally view $N_\Delta$ as a compactification of $N_\R$ by strata 
that are infinitely far away. 
\end{Ex}

\begin{defn}
Let $\CC$ be a one dimensional $\Lambda$-rational polyhedral complex in $N_\R$. 
For an edge $e$ (i.e.~a one-dimensional polyhedron) of $\CC$ we denote by 
$\Linear(e) = \{\lambda (u_1 - u_2) \mid u_1, u_2 \in e, \lambda \in \R \}$
the \emph{linear space of e}. 
Since $X$ is $\Lambda$-rational, $\Linear(e)$ contains a canonical lattice which we denote by $\Z(e)$. 

For a vertex $v$ of $e$ we denote by $w_{v,e}$ the unique generator of $\Z(e)$ that points
in $e$ away from $v$.

We call $\CC$ \emph{weighted} if every edge is equipped with a positive integral weight $m(e)$ and 
\emph{balanced} if for every vertex $v$ of $\CC$ we have
\begin{align*}
\sum_{e \colon v \prec e} m(e) w_{v,e} = 0.
\end{align*}
The \emph{local cone} at $v$ is the one-dimensional fan whose rays are spanned by the $w_{w,e}$ 
and given weight $m(e)$ for $v \prec e$. 
This is also sometimes referred to as the \emph{star} of the vertex $v$ (see for example \cite{MacStu:book}). 

\end{defn}

\begin{defn} \label{defn tropical curve}
A \emph{tropical curve in $N_\R$} is a one dimensional $\Lambda$-rational polyhedral complex 
equipped with weights on its edges that satisfies the 
balancing condition, up to the equivalence relation generated by subdivision of edges preserving the weights. 

A \emph{tropical curve $X$ in a tropical toric variety $N_\Delta$} is the 
closure in $N_\Delta$ of a tropical curve $X^\circ$  in $N_\R$. 
\end{defn}

$X \setminus X^\circ$ is a finite set, whose points we consider as vertices of $X$ 
and call the \emph{infinite vertices}. 
The edges of $X$ are the closures of the edges of $X^\circ$.

A \emph{$\Lambda$-metric graph} (which we will often just call a \emph{metric graph}) 
is, roughly speaking, a finite graph in which every edge $e$ has a positive length $l_e \in \Lambda \cup \{\infty\}$.
We allow loop edges, meaning edges whose endpoints agree and half open edges, 
i.e.~edges which have only one endpoint. 
If $l_e \in \Lambda_{>0}$, we view $e$ as an interval of length $e$. 
Half open edges are identified with $\R_{\geq 0}$. 
Leaf edges are the only edges that are allowed to have infinite length and are identified with $[0, \infty]$ with
the topology of a closed interval. 
For a more precise account on metric graphs, we refer the reader to \cite[Section 2.1]{ABBR}. 

By an \emph{edge} of a metric graph $\Gamma$ we mean an edge in some graph model $G$ of $\Gamma$. 
For an edge $e$ of $\Gamma$ we denote by $\mathring e$ the relative interior of $e$, 
meaning $e$ with its endpoints removed. 
For two points $x,y \in \mathring e$ we denote by $d_e(x,y)$ their distance in $\mathring e$. 
(Note that this might not be the distance in $\Gamma$, as there might be a shorter path that leaves $\mathring e$.)

We call a metric graph \emph{finite} if all its edges have finite length.

\begin{Ex}
A tropical curve in $N_\R$ has a canonical structure as a metric graph 
where the length of an edge is given by the \emph{lattice length},
meaning the length of the primitive vector $w_{v,e}$ equals $1$. 
\end{Ex}

A tropical curve $X$ in a tropical toric variety $N_\Delta$ is not necessarily 
a metric graph since two infinite rays might meet at infinity, 
creating a vertex at infinity which does not have valence $1$. 
However, $X$ is a metric graph if every point in $X \setminus X^\circ$ has exactly one adjacent edge.

\begin{defn} \label{defn smooth}
An edge in a tropical curve is \emph{smooth} if its weight is $1$. 
A finite vertex $v$ is \emph{smooth} if 
$\langle w_{v,e} \mid v \prec e \rangle_\Z$ is a saturated lattice of rank $\val(v)-1$ in $N$, 
where $\val(v)$ is the number of edges adjacent to $v$. 
An infinite vertex is smooth if it has one adjacent edge. 
A vertex that is not smooth is called \emph{singular}. 
A tropical curve is \emph{smooth} if all its edges and vertices are smooth. 
\end{defn}

\begin{bem}
Following \cite{IKMZ} a tropical variety is \emph{smooth} if it 
is locally isomorphic to the Bergman fan of a matroid. 

A one-dimensional weighted fan in $\R^n$ is the Bergman fan of a matroid if
and only if it is isomorphic to the fan whose rays are spanned by 
$x_1,\dots,x_n$ and $- \sum_{i=1}^n x_i$ and all weights are $1$. 
Thus Definition \ref{defn smooth} agrees with the one in \cite{IKMZ} for the case of curves. 
\end{bem}

\begin{Ex}
Consider the tropical curves in Figure \ref{Figure smooth}. 
Each of them depicts a vertex in a tropical curve in $\R^2$ with lattice $N = \Z^2$.
In the leftmost picture, the outgoing directions are $(-1,0), (0,-1)$ and $(1,1)$, 
which span $\Z^2$, thus $v_1$ is a smooth vertex. 
In the picture in the middle, the span of the primitive vectors is again $\Z^2$, but
there are $4$ vertices adjacent to $v_2$, thus $v_2$ is not smooth. 
In the picture on the right, the outgoing directions are $(2,-1), (-1,2)$ and $(-1,-1)$. 
The span of these vectors is $\{ (x,y) \in \Z^2 \mid x-y \text{ is divisible by } 3\}$. 
This has rank $2$, but is not saturated in $\Z^2$, thus $v_3$ is not 
a smooth vertex. 
\end{Ex}

\begin{figure}
\begin{tikzpicture}
%smooth vertex
\draw (0,0) -- (-1,0);
\draw (0,0) -- (0,-1); 
\draw (0,0) -- (1,1); 

\node [above] at (0,0) {$v_1$}; \fill (0,0) circle (2pt);

%4 valent vertex
\draw (4,-1) -- (4,1); 
\draw (3,0) -- (5,0); 

\node[below] at (4.2,0) {$v_2$}; \fill (4,0) circle (2pt);

%lattice index
\draw (8,0) -- ++(2,-1); 
\draw (8,0) -- ++(-1, 2); 
\draw (8,0) -- ++(-1,-1); 

\node[right] at (8,0) {$v_3$}; \fill (8,0) circle (2pt); 

\end{tikzpicture}
\caption{The types of vertices in tropical curves in $\R^2$. 
The vertex on the left is smooth, the other two vertices are not smooth.}
\label{Figure smooth}
\end{figure}
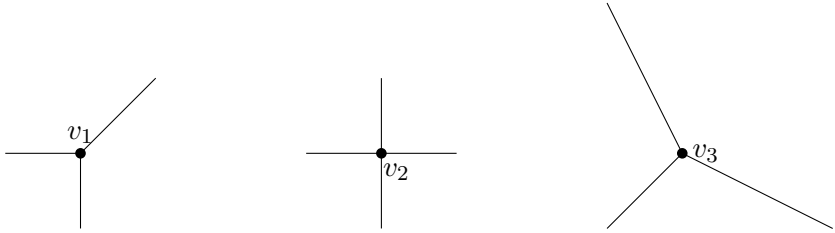

%--------------------------------------------------------------------------------------------------------------------------------------

\subsection{Berkovich curves and their extended skeleta}

Let $X$ be a variety over $K$. 
The associated Berkovich space \cite{BerkovichSpectral} is 
\begin{align*}
\Xan := \{ x = (p_x, \abs_x) \mid p_x \in X, \abs_x 
\text{ is an absolute value on } k(p_x) \text{ extending } \abs_K\}
\end{align*}
with the topology such that the canonical forgetful map $\Xan \to X$ is continuous and 
for all open subsets $U$ of $X$ and $f \in \Ocal(U)^\times$ the map $U^{\an} \to \R, (p_x, \abs_x) \mapsto \vert f(p_x) \vert_x$ is continuous. 
We will often write $\vert f(x) \vert := \vert f(p_x) \vert_x$. 
If $X = \Spec(A)$ is an affine variety then 
\begin{align*}
\Xan = \{\abs \text{ multiplicative seminorm on } A \text{ extending }  \abs_K\}
\end{align*}
with the topology such that for all $f \in A$ the map $\Xan \to \R; \abs \mapsto \vert f \vert$ is continuous. 
For morphism $\varphi \colon X \to Y$ of $K$-varieties we obtain a morphism $\varphi^{\an} \colon \Xan \to Y^{\an}$. 

Now let $X$ be a curve over $K$.
For $x  \in \Xan$ we denote by $\Hscr(x)$ the completion of $k(p_x)$ with respect to $\abs_x$
and by $\Htilde$ its residue field. 
Following Berkovich and Thuillier \cite{BerkovichSpectral, Thuillier}
we say $x$ is of type I if $p_x \in X(K)$ and 
of type II if $p_x$ is the generic point of $X$ and $\trdeg [ \Htilde : \Ktilde ] = 1$. 
If $x$ is of type I, then $\abs_x = \abs_K$, thus 
the forgetful map $\Xan \to X$ induces a bijection from the set of type I points of $\Xan$
onto $X(K)$. 
We will thus identify $X(K)$ with the subset of $\Xan$ that consists of type I points. 
If $x$ is of type II, then we denote by $C_x$ the smooth projective $\Ktilde$-curve with function field $\Htilde$
and by $g(x)$ its genus, which we call the \emph{genus of $x$}.

We now recall the notion of completed skeleta of $\Xan$,
which is due to Baker, Payne and Rabinoff \cite{BPR2}.

\begin{defn}
We consider $\A^1 = \Spec K [T]$. 
For $- \infty \leq s < r \in \R$ denote 
\begin{align*}
B(r) = \{ x \in \A^{1, \an} \mid \log \vert T \vert_x < r\} \text{ and } 
A(r,s) = \{ x \in \A^{1, \an} \mid s < \log \vert T \vert_x < r \}.
\end{align*}
We call $B(r)$ an \emph{open disc} of logarithmic radius $r$ and 
$A(r,s)$ a \emph{generalized open annulus} of logarithmic radii $s$ and $r$. 
We call $A(r,s)$ an \emph{annulus} with logarithmic radii $s$ and $r$ if $s \in \R$
and a \emph{punctured disc} of radius $r$ if $s = -\infty$. 
We call $r - s \in \R \cup \infty$ the \emph{length of A(r,s)}. 

We denote by $\rho_{B(t)}$ the element of $B(t)$ defined by 
$\left \vert \sum a_i T^i \right \vert_{\rho_{B(t)}} = \max_i \vert a_i \vert t^i$ 
and call the set 
\begin{align*}
\Sigma(A(r,s)) = \left\{ \rho_{B(t)}  \mid s < \log t < r \right\} 
\end{align*}  
the \emph{skeleton} of $A(r,s)$.
There is a canonical retraction $\tau \colon A(r,s) \to \Sigma(A(r,s))$ which is a strong deformation retraction. 
\end{defn}

\begin{defn}
Let $X$ be a smooth projective curve over $K$. 
A \emph{completed semistable vertex set} $V$ of $X$ is a finite subset of $\Xan$ 
consisting of type I and II points such that $\Xan \setminus V$
is isomorphic to a disjoint union of finitely many generalized open annuli and 
infinitely many open discs. 
\end{defn}

For a completed semistable vertex set $V$ of $\Xan$ there is a canonical associated subspace $\Sigma(V)$ of $\Xan$, 
called the \emph{completed skeleton} $\Sigma(V)$,
which is a metric graph.
There is a canonical retraction $\tau_V \colon \Xan \to \Sigma(V)$, 
such that $\Sigma(V)$ is a strong deformation retract of $\Xan$. 
As the name suggests, the vertex set of $\Sigma_V$ is $V$. 
The edges are the skeleta of the generalized open annuli 
that are connected components of $\Xan \setminus V$. 
The length of such an edge is the length of the corresponding annulus. 

If $X$ is projective and $V$ is a completed semistable vertex set that only consists of type II points, 
we call $V$ a \emph{semistable vertex set} and $\Sigma(V)$ a \emph{finite skeleton} of $X$. 
A finite skeleton is a finite metric graph and we will often denote it by $\Gamma$. 

Let $V$ be a completed semistable vertex set of $X$. 
Then the set of type II points in $V$ forms a semistable vertex set for $X$. 
We call the associated finite skeleton the \emph{finite part} of $\Sigma(V)$ and denote it by $\Sigma(V)_{\fin}$.

\begin{defn} \label{defn Mumford}
A smooth projective curve of genus $g > 0$ is called \emph{Mumford curve}
if for some semistable vertex set $V$ the skeleton $\Gamma(V)$ 
has first Betti number equal to $g$.
\end{defn}

\begin{bem} \label{bem Mumford}
Note that since $\Gamma(V)$ is a deformation retract of $\Xan$, 
the first Betti number of $\Gamma(V)$ is independent of $V$. 
Thus we might replace ``some'' by ``every'' in Definition \ref{defn Mumford}. 
Furthermore $X$ is a Mumford curve if and only if $g(x) = 0$ for all type II points $x$ in $\Xan$. 
Another equivalent definition of Mumford curve is that any point $x \in \Xan$ has a neighborhood 
that is isomorphic to an open subset of $\Pbb^{1,\an}$ \cite[Proposition 2.26 \& Theorem 2.28]{JellWanner}. 
\end{bem}

%--------------------------------------------------------------------------------------------------------------------------------------

\subsection{Tropicalization of curves}

Let $Y$ be a toric variety with dense torus $T$. 
Let $N$ be the cocharacter lattice of $T$, $N_\R := N \otimes \R$ and 
$\Delta$ the fan in $N_\R$ associated to $Y$. 

\begin{defn}
The \emph{tropicalization of Y} is 
\begin{align*}
\Trop(Y) := N_\Delta.
\end{align*}
\end{defn}

There is a canonical tropicalization map $\trop \colon Y^{\an} \to \Trop(Y)$, 
which is a continuous proper map of topological spaces \cite[Section 3]{Payne}. 

We assume that the reader is familiar with tropicalizations of closed subvarieties 
of algebraic tori \cite{MaclaganSturmfels, Gubler2}. 
Here we consider tropicalizations of closed subvarieties of toric varieties, 
which may be seen as a compactification of the latter. 
We quickly sketch the relation: 
Given a closed embedding $\varphi \colon X \to Y$ of a smooth projective curve X into
a toric variety $Y$ that meets the dense torus $T$, denote by $X^\circ := \varphi^{-1}(T)$. 
Then $\Trop_\varphi(X^\circ)$ is a dense open subset of $\Trop_\varphi(X)$ 
and we obtain the latter from the former by putting points at the end of the unbounded edges. 

\begin{Ex}
If $Y = \G_m^n$ is a torus of dimension $n$ with fixed coordinates, then $\Delta$ is only the origin in $\R^n$ and 
we have $\Trop(Y) = \R^n$. 
The restriction of the map $\trop \colon \G_m^{n, \an} \to \R^n$ to $\G_m^n(K) = (K^*)^n$ is the
usual tropicalization map $X(K) \to \R^n; x \mapsto (\log \vert x_1 \vert_K,\dots,\log \vert x_n \vert_K)$. 

If $Y = \Pbb^1$, then Example \ref{example tropical toric} shows that $\Trop(\Pbb^1)$ 
is homeomorphic to a closed interval.
Since it contains a one-dimensional vector space as a dense open subset, a good point of view is
$\Trop(\Pbb^1) = [-\infty, \infty]$ with the topology of a closed interval. 

The map $\trop \colon \Pbb^{1, \an} \to \Trop(\Pbb^1)$ is then given by 
$(p, \abs_x) \mapsto \log \vert z(p) \vert_x$, where $z$ is the coordinate function on $\Pbb^1$.  
\end{Ex}

\begin{bem}
For two toric varieties $Y_1$ and $Y_2$, we have
$\Trop(Y_1 \times Y_2) = \Trop(Y_1) \times \Trop(Y_2)$. 
This holds because the fan of $Y_1 \times Y_2$ is the product of the fans of $Y_1$ and $Y_2$. 
\end{bem}

Let $X$ be a curve over $K$.
For a closed embedding $\varphi \colon X \to Y$ we denote $\varphi_{\trop} := \trop \circ \varphi^{\an}$ and 
$\Trop_\varphi(X) := \varphi_{\trop}(\Xan)$ 
the associated tropicalization of $X$. 
One can define canonical weights on $\Trop_{\varphi}(X)$ 
that make it into a tropical curve in $\Trop(Y)$ in the sense of Definition \ref{defn tropical curve} 
(see for example \cite{Gubler2}).
We will define these weights in Definition \ref{weight}. 
 
\begin{defn} \label{defn refinement}
If $Y'$ is another toric variety,
$\varphi' \colon X \to Y'$ is another closed embedding and $\pi \colon Y' \to Y$ is a morphism of toric varieties, 
there exists a canonical map $\Trop(Y') \to \Trop(Y)$, which is linear on the dense subset $N_\R$ 
and maps $\Trop_{\varphi'}(X)$ onto $\Trop_{\varphi}(X)$. 
We call $\varphi'$ a \emph{refinement} of $\varphi$. 
\end{defn}

Note that refinements yield the inverse system in Payne's result that the inverse 
limit of all tropicalizations is homeomorphic to $\Xan$ \cite[Theorem 4.2]{Payne}.

%--------------------------------------------------------------------------------------------------------------------------------------

\subsection{Factorization skeleta} \label{extended skeleton of an embedding}

Let $\varphi \colon X \to Y$ be a closed embedding of a smooth projective curve $X$ into a toric variety $Y$
that meets the dense torus $T$.  
Denote by $X^\circ := \varphi^{-1}(T)$ the preimage of the dense torus.

\begin{defn} \label{defn minimal skeleton}
Let $\Sigma(\varphi)$ be the set of points in $\Xan$ that do not have 
an open neighborhood that is isomorphic to an open disc and 
contained in $(X^\circ)^{\an}$. 
We call $\Sigma(\varphi)$ the \emph{completed skeleton associated to $\varphi$}. 
\end{defn}

The set $\Sigma(\varphi)$ is indeed a completed skeleton for $X$ \cite[Theorem 4.22]{BPR2}. 
We denote by $\tau_\varphi \colon \Xan \to \Sigma(\varphi)$ the retraction.

Baker, Payne and Rabinoff show that we have a commutative diagram 
\begin{align} \label{dia factorization}
\begin{xy}
\xymatrix{
\Xan \ar[rr]^{\varphi_{\trop}} \ar[rd]_{\tau_{\varphi}}  && \Trop_\varphi(X) \\
& \Sigma(\varphi) \ar[ru]_{\varphi_{\trop}|_{\Sigma(\varphi)}} 
}
\end{xy}
\end{align}
and that $\varphi_{\trop}|_{\Sigma(\varphi)}$ is linear on 
each edge of $\Sigma(\varphi)$ \cite[Lemma 5.3 \& Proposition 5.4 (1)]{BPR}.

We can subdivide $\Trop_{\varphi}(X)$ and $\Sigma(\varphi)$ in such a way that each edge of $\Sigma(\varphi)$ 
is either contracted to a point or 
mapped homeomorphically to an edge of $\Trop_{\varphi}(X)$ \cite[Lemma 5.4. (2)]{BPR}. 
Let $e$ be an edge in $\Trop_\varphi(X)$. 
Let $e_1,\dots,e_k$ be the edges of $\Sigma(\varphi)$ mapping homeomorphically to $e$. 
For each $i$, we fix $x_i \neq y_i \in \mathring e_i$.

\begin{defn} \label{weight}
We call  
\begin{align*} 
m(e_i) = \frac {d_e(\varphi_{\trop}(x_i),\varphi_{\trop}(y_i))} {d_{e_i}(x_i,y_i)} 
\text{ and } m(e) = \sum_{i =1}^k m(e_i)
\end{align*}
the \emph{stretching factor} of $\varphi_{\trop}|_{e_i}$ and \emph{the weight of $e$}, respectively.
\end{defn}
 
The definition of weight agrees with the usual one 
(see for example \cite[Definition 3.14]{Gubler2}) by \cite[Corollary 5.9]{BPR}.

\begin{prop} \label{prop defn fully faithful}
Let $\varphi \colon X \to Y$ be a closed embedding of $X$ into a toric variety that meets the dense torus $T$
and $\Sigma(\varphi)$ the associated skeleton. 
Denote by $X^{\circ} := \varphi^{-1}(T)$. 
Then the following are equivalent: 
\begin{enumerate}
\item
$\varphi_{\trop}$ maps $\Sigma(\varphi)$ homeomorphically onto its image and is an 
isometry when restricted to $\Sigma(\varphi) \cap X^{\circ, \an}$.
\item
The map $\varphi_{\trop}|_{\Sigma(\varphi)} \colon \Sigma(\varphi) \to \Trop_{\varphi}(X)$ 
is injective and all weights on $\Trop_{\varphi}(X)$ are $1$. 
\end{enumerate}
\end{prop} 
\begin{proof}
Assume that $ii)$ holds. 
The map $\varphi_{\trop}|_{\Sigma(\varphi)}$ is surjective, thus bijective.
Since it is a bijective map between compact Hausdorff spaces, it is a homeomorphism. 
Hence both $i)$ and $ii)$ imply that $\varphi_{\trop}|_{\Sigma(\varphi)}$ is a 
homeomorphism onto its image. 

Thus it remains to show that if  
$\varphi_{\trop}|_{\Sigma(\varphi)}$ 
is a homeomorphism it is an isometry when restricted to  $\Sigma(\varphi) \cap X^{\circ, \an}$ 
if and only if all weights on $\Trop_{\varphi}(X)$ are all equal to one. 
This follows from Definition \ref{weight}. 
\end{proof}

\begin{defn} 
We say that $\varphi_{\trop}$ is \emph{fully faithful} 
if the equivalent conditions of Proposition \ref{prop defn fully faithful} hold.  
\end{defn}

The notion of fully faithful tropicalization is stronger then the notion of faithful tropicalization
introduced by Baker, Payne and Rabinoff \cite{BPR}. 
It is also slightly stronger then the notion of totally faithful tropicalization introduced by 
Cheung, Fantini, Park and Ulirsch \cite{CFPU}
(see also \cite{CDMY}).
The difference is that a totally faithful tropicalization only needs to be an isometry when restricted to 
$\Sigma(\varphi) \cap X^{\circ, \an}$.
Note however that the authors of \cite{CFPU} mainly work in the situation of tropical compactifications and 
in this case 
the notions of totally faithful and fully faithful agree.

%--------------------------------------------------------------------------------------------------------------------------------------

\subsection{Rational functions and divisors on metric graphs} \label{functions and divisors}

Let $\Gamma$ be a finite $\Lambda$-metric graph. 
A point $x \in \Gamma$ is called \emph{$\Lambda$-rational} 
if its distance from some, or equivalently every, vertex is in $\Lambda$.
A \emph{rational function} on $\Gamma$ is a piecewise linear function $F \colon \Gamma \to \R$ with integer slopes
all of whose points of non-linearity are $\Lambda$-rational. 
A \emph{divisor} on $\Gamma$ is a finite formal linear combination of $\Lambda$-rational points. 
Its \emph{degree} is the sum of the coefficients. 
We denote by $\Div(\Gamma)$ the group of divisors. 
For a rational function $F$ its divisor is 
\begin{align*}
\divisor(F) := \sum \lambda_i x_i \text{ where } 
\lambda_i := \sum_{e \colon x_i \prec e} d_eF(x_i) 
\end{align*}
and $d_eF(x_i)$ is the outgoing slope of $F$ along the edge $e$ at $x_i$. 
We call $\divisor(F)$ a \emph{principal divisor} on $\Gamma$. 
We denote by $\Prin(\Gamma)$ the group of principal divisors on $\Gamma$. 

Let $X$ be a smooth projective curve and $\Gamma$ a finite skeleton with retraction $\tau$. 
Let $f$ be in $K(X)^*$. 
Then $F :=  \log \vert f(x) \vert \big \vert_{\Gamma}$ is a rational function on $\Gamma$ and 
$\tau_* (\divisor(f)) = \divisor(F)$ \cite[Theorem 5.15]{BPR2} 
(see also \cite[Proposition 3.3.15]{Thuillier} for the same result phrased in a slightly different language).

\begin{defn}
We say that edges $e_1,\dots,e_g$ \emph{form the complement of a spanning tree of $\Gamma$}
if there exists a graph model $G$ for $\Gamma$ with set of edges $E$ 
such that $e_i \in E$ and the subgraph of $G$ spanned by the edges $E \setminus \{e_1,\dots,e_g\}$
is connected, contractible and contains all vertices of $G$. 
\end{defn}
 
Note that in this definition, $g$ is necessarily the first Betti number of $\Gamma$.

The notion of break divisor was introduced by Mikhalkin, and Zharkov \cite{MikZharII}. 
They observed that any degree $g$ divisor on a metric graph has a unique 
break divisor in its rational equivalence class (see Theorem \ref{MikhalkinZharkov}).
Break divisors were also studied in detail by An, Baker, Kuperberg, and Shokrieh, 
who also study discrete versions \cite{ABKS}.

\begin{defn}
Let $\Gamma$ be a metric graph and $g = \dim_\R \HH^1(\Gamma, \R)$ its first Betti number. 
A \emph{break divisor} is a degree $g$ effective divisor $B = p_1 + \dots + p_g$ such that 
there exist edges $e_1,\dots,e_g$ that form the complement 
of a spanning tree of $\Gamma$ such that $p_i \in e_i$. 
\end{defn} 

\begin{satz} [Mikhalkin - Zharkov] \label{MikhalkinZharkov}
Let $D$ be a degree $g$ divisor on $\Gamma$. 
Then there exists a unique break divisor $B$ on $\Gamma$ such that $D - B \in \Prin(\Gamma)$. 
\end{satz}

Break divisors will play an important role in Theorem \ref{lifting theorem}, 
which we will use to prove our main theorems, as well as to construct 
tropicalizations in honeycomb form for elliptic curves (see Example \ref{tate curves}). 
In our applications we will deal with break divisors that are supported on two-valent points of $\Gamma$. 
If $B$ is such a break divisor then $\Gamma \setminus \supp(B)$ is connected and contractible. 

We will see in Example \ref{counterexample} that it is really necessary to restrict to break divisors
that are supported on two-valent points in Theorem \ref{lifting theorem}.

%--------------------------------------------------------------------------------------------------------------------------------------

\section{Lifting theorem}

In this section $X$ is a smooth projective Mumford curve of genus $g$ over $K$. 
We fix a semistable vertex set $V$ with corresponding finite skeleton $\Gamma$ and 
retraction $\tau$. 
We denote by $J_0(X) := \{ [D] \in \Pic(X) \mid \tau_* D \in \Prin(\Gamma) \}$.

\begin{prop} \label{prop BR lifting}
Let $B = p_1 + \ldots + p_g$ be a break divisor on $\Gamma$ that is supported on two-valent points and 
write $R_i = \tau^{-1}(p_i) \cap X(K)$.
Then for all $\YS = (y_1,\dots,y_g) \in  R_1 \times \dots \times R_g$ the map
\begin{align*}
\varphi_{\YS} \colon R_1 \times \ldots \times R_g &\to J_0(X) \\ 
 (x_1,\dots,x_g) &\mapsto  \sum_{i=1}^g \left[x_i-y_i \right] 
\end{align*}
is a surjection.
\end{prop}
\begin{proof}
We consider \cite[Proof of Theorem 1.1]{BRab}.
Baker and Rabinoff work in the same setup, but for them $X$ is any curve, not necessarily a Mumford curve. 
Thus in their situation both the set of $\YS$ they allow and the domain of $\varphi_{\YS}$ is 
$(R_1 \times \ldots \times R_b) \times C^*$. 
Here $b$ is the first Betti number of the skeleton of $X$ and 
$C^* = \prod_{x \in \Xan; g(x) > 0} C_x(\Ktilde)^{g(x)}$. 
An element $\YS \in (R_1 \times \ldots \times R_b) \times C^*$ is denoted by 
$(\YS_1, \YS_2)$ for $\YS_1 \in R_1 \times \ldots \times R_b$ and $\YS_2 \in C^*$. 
They show that $\varphi_{(\YS_1,\YS_2)}$ is surjective when $\YS_2$ is generic. 
If $X$ is a Mumford curve, then $b = g$ and $C^*$ is just a one point set. 
Thus $\YS_2$ is automatically generic and our proposition follows.  
\end{proof}

\begin{satz} \label{lifting theorem}
Let $D \in \Div(X)$ of degree $g$ and $B = p_1 + \dots + p_g \in \Div(\Gamma)$ a break divisor 
such that $\tau_* D - B$ is a principal divisor on $\Gamma$. 
Assume that $B$ is supported on two-valent points of $\Gamma$. 
Then there exist $x_i \in X(K)$ such that $\tau_* x_i = p_i$ and such that 
$D - \sum_{i=1}^g  x_i$ is a principal divisor on $X$. 
\end{satz}

\begin{proof}
Let $y_i \in X(K)$ such that $\tau_* y_i = p_i$.
We have $\left[ D - \sum_{i=1}^g y_i \right] \in J_0(X)$. 
Thus by Proposition \ref{prop BR lifting} there exist $x_i \in \tau^{-1}(p_i) \cap X(K)$ such that 
$\left[ D - \sum_{i=1}^g y_i \right] = \left[ \sum_{i=1}^g (x_i - y_i) \right]$. 
In other words $\left[ D - \sum_{i=1}^g x_i \right] = 0$ 
which means that $D - \sum_{i=1}^g x_i$ is a principal divisor on $X$.   
\end{proof}

\begin{defn}
Let $e$ be an edge of $\Gamma$. 
Four points $p_1,p_2,p_3,p_4 \in \mathring e$ are called \emph{pillar points in $e$}
if they are $\Lambda$-rational, $d_e(p_1, p_2) = d_e(p_3, p_4)$ and
for $i = 2, 3$ we have $[p_{i-1},p_{i}] \cap [p_{i}, p_{i+1}] = p_i$. 
(See Figure \ref{figure pillar points}.)
\end{defn} 

\begin{figure}
\begin{tikzpicture}
\draw (-1,0.3) --  (0,0.3) -- (1,1.3) -- (3,1.3) -- (4,0.3) -- (5,0.3);
\draw (-1,0) -- (0,0) -- (1,0) -- (3,0) -- (4,0) -- (5,0);
\node [below] at (0,0) {$p_1$}; \fill (0,0) circle (2pt);
\node[below] at (1,0) {$p_2$}; \fill (1,0) circle (2pt);
\node[below] at (3,0) {$p_3$}; \fill (3,0) circle (2pt);
\node[below] at (4,0) {$p_4$}; \fill (4,0) circle (2pt);
\node[right] at (5,0) {$e$};
\end{tikzpicture}
\caption{An edge $e$ with four pillar points $p_1,p_2,p_3$ and $p_4$ 
and a piecewise linear function with divisor $p_1 - p_2 - p_3 + p_4$.}
\label{figure pillar points}
\end{figure}
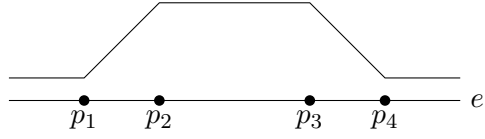

Figure \ref{figure pillar points} shows the graph of a piecewise linear function whose divisor is
$p_1 - p_2 - p_3 + p_4$. In particular that divisor is principal. 

\begin{kor} \label{lifting corollary}
Let $D \in \Div^0(X)$ such that $\tau_* D$ is a principal divisor on $\Gamma$.
Let $e_1,\dots,e_g$ be edges that form the complement of a spanning tree of $\Gamma$. 
Fixing pillar points $p_{i,1},p_{i,2},p_{i,3},p_{i,4}$ in $\mathring e_i$ 
there exist $x_{ij} \in X(K)$ such that $\tau(x_{ij}) = p_{ij}$ and $f \in K(X)^*$ such that 
$\divisor(f) = D + \sum_{i=1}^g (x_{i,1} + x_{i,4}) - \sum_{i=1}^g (x_{i,2} + x_{i,3})$.
\end{kor}
\begin{proof}
The divisor $\sum_{i=1}^g (p_{i,1} + p_{i,4}) - \sum_{i=1}^g (p_{i,2} + p_{i,3})$ 
is principal on $\Gamma$, thus so is 
$\tau_*D + \sum_{i=1}^g (p_{i,1} + p_{i,4}) - \sum_{i=1}^g (p_{i,2} + p_{i,3})$. 
Thus, for $j = 1, 3,4$, fixing $x_{ij}$ such that $\tau_* x_{ij} = p_{ij}$ and writing 
$D' = D + \sum_{i=1}^g (x_{i,1} + x_{i,4})  - \sum_{i=1}^g x_{i,3}$ and 
$B = p_{1,2} + \dots + p_{g,2}$, we find that  
$\tau_* D' - B$ is a principal divisor on $\Gamma$. 
Since $B$ is a break divisor supported on two-valent points, applying Theorem \ref{lifting theorem} to 
$D'$ and $B$ we get the result.
\end{proof}

\begin{Ex} [Tate curves] \label{tate curves}
Chan and Sturmfels use theta functions to produce nice 
tropicalizations of elliptic curves \cite{ChanSturmfels} (see also \cite[Theorem 6.2]{BPR}). 
In this example we show how Theorem \ref{lifting theorem} can 
be used to construct such nice tropicalizations combinatorially. 

Let $E$ be an elliptic curve with bad reduction. 
We will use Theorem \ref{lifting theorem} to construct a closed embedding $\varphi \colon E \to \Pbb^2$ whose 
tropicalization looks like the right hand side of Figure \ref{figure tate curve}, 
which Chan and Sturmfels call \emph{symmetric honeycomb form}.

The minimal skeleton $\Gamma_{\min}$ is a circle. 
We pick three points $q_1, q_2, q_3 \in \Gamma_{\min}$ that are equidistant from each other. 
Our skeleton $\Gamma$ is obtained from $\Gamma_{\min}$ 
by adding edges of length $d(q_i, q_j)/2$ at each of the $q_i$, 
denoting their endpoints by $p_i$. 
We subdivide each edge $[q_i,q_j]$ at its midpoint
and label our new vertices as on the 
left hand side of Figure \ref{figure tate curve}. 
The solid part of the figure is now our skeleton $\Gamma$.

We pick points $x_{1,1} \neq x_{1,2}, x_{2,1} \neq x_{2,2}, x_{3,1} \neq x_{3,2}$ and 
$x_{6} \in E(K)$ such that $\tau(x_{i,j}) = p_i$ and $\tau(x_6) = p_6$. 

Let $D_1 = - x_{1,1} + x_{2,1} - x_{2,2} + x_{3,1} - x_6$.
Then $\tau_* D_1 = -p_1 + p_3 - p_6$
and $\tau_* D_1 + p_4 = \divisor(F_1)$ for a rational function $F_1$ on $\Gamma$. 

Now applying Theorem \ref{lifting theorem} to $-D_1$ and $p_4$ 
we obtain a function $f_1 \in K(E)^*$ and $x_4 \in E(K)$ such that $\tau(x_4) = p_4$ and 
$\divisor(f_1) = D_1 + x_4$.
We normalize $f_1$ such that $F_1 = \log \vert f_1 \vert \big\vert_\Gamma$. 

Similarly let 
$D_2 = -x_{1,1} + x_{1,2} - x_{2,2} + x_{3,2} - x_6$
then $\tau_* D_2 = - p_2 + p_3 - p_6$
and $\tau_* D_2 + p_5 = \divisor(F_2)$, for a rational function $F_2$ on $\Gamma$.
We obtain a function $f_2 \in K(E)^*$ and $x_5 \in E(K)$ such that 
$\tau(x_5) = p_5$ and $\divisor(f_2) = D_2 + x_5$. 

Let $\varphi$ be the morphism associated to the rational map $[f_1: f_2 : 1] \colon E \to \Pbb^2$.  
By construction, the graph on the left hand side of Figure \ref{figure tate curve}, including 
the dashed lines, which are infinite edges, is the associated completed skeleton $\Sigma(\varphi)$. 
We write $G_i = \log \vert f_i \vert \big \vert_{\Sigma(\varphi)}$. 
Note that $G_i|_{\Gamma} = F_i$.
Further, $\varphi_{\trop}|_{\Sigma(\varphi)} = (G_1,G_2)$.
Thus $\Trop_{\varphi}(E) = (G_1,G_2)(\Sigma(\varphi))$ is the tropical curve 
on the right hand side of Figure \ref{figure tate curve}. 

The functions $f_1, f_2, 1$ are linearly independent over $K$, since $f_1$ is 
not constant on the zeros of $f_2$. 
Thus by the Riemann-Roch theorem, they form a basis of $L(D)$ where $D = x_{1,1} + x_{2,2} + x_6$. 
Since $D$ is very ample by \cite[Corollary IV.3.2(b)]{Hartshorne}, this shows that $\varphi$ is a closed embedding.
\end{Ex}

\begin{figure} 
\begin{tikzpicture} 

%Circle
\draw (0,0) circle (1);

%Additional Finite Edges
\draw (1,0) -- ++(1,0);
\draw (-0.5, 0.866) -- ++(-0.5, 0.866); 
\draw (-0.5, -0.866) -- ++(-0.5, -0.866); 

%Dashed edges from the circle
\draw [dashed] (-1,0) -- ++ (-1,0); 
\draw [dashed] (0.5, 0.866) -- ++(0.5, 0.866);
\draw [dashed] (0.5, -0.866) -- ++(0.5, -0.866);

%Other dashed edges
 \draw [dashed] (2,0) -- ++(0.5,0.5); \draw [dashed] (2,0) -- ++(0.5,-0.5); 
 \draw [dashed] (-1, 1.732) -- ++(-0.8, 0); \draw [dashed] (-1,1.732) -- ++(0,0.8); 
 \draw [dashed] (-1, -1.732) -- ++(-0.8, 0); \draw [dashed] (-1,-1.732) -- ++(0,-0.8);

%q1q2q3
\node [left] at (1,0) {$q_1$};                    \fill (1,0) circle (2pt); 
\node [below] at (-0.5, 0.866) {$q_2$};        \fill (-0.5, 0.866) circle (2pt);
\node [above] at (-0.5, -0.866) {$q_3$};         \fill (-0.5, -0.866) circle (2pt);

%p1p2p3
\node [right] at (2,0) {\;$p_1$};                  \fill (2,0) circle (2pt); 
\node [right] at (-1, 1.732) {$p_2$};        \fill (-1, 1.732) circle (2pt); 
\node [right] at (-1, -1.732) {$p_3$};         \fill  (-1, -1.732) circle (2pt);

%p4p5p6
\node [right] at (-1,0) {$p_4$};                    \fill (-1,0) circle (2pt); 
\node [above] at (0.5, -0.866) {$p_5$};        \fill (0.5, -0.866) circle (2pt); 
\node [below] at (0.5, 0.866) {$p_6$};         \fill (0.5, 0.866) circle (2pt);

%thexi
\node [right] at (2.5, 0.5) {$x_{1,1}$};          \fill (2.5,0.5) circle (2pt);     
\node [right] at (2.5, -0.5) {$x_{1,2}$};           \fill (2.5,-0.5) circle (2pt);
\node [above] at (-1,2.532) {$x_{2,2}$};          \fill (-1,2.532) circle (2pt);
\node [left] at (-1.8,1.732) {$x_{2,1}$};           \fill (-1.8,1.732) circle (2pt);
\node [below] at (-1,-2.532) {$x_{3,2}$};          \fill (-1,-2.532) circle (2pt);
\node [left] at (-1.8,-1.732) {$x_{3,1}$};           \fill (-1.8,-1.732) circle (2pt);
\node [right] at (1, -1.732) {$x_5$};        \fill (1, -1.732) circle (2pt); 
\node [left] at (-2,0) {$x_4$};                    \fill (-2,0) circle (2pt);
\node [above] at (1, 1.732) {$x_6$};               \fill (1,1.732) circle (2pt);

\draw  (6,-1) -- ++(1,0) -- ++(1,1) -- ++(0,1)-- ++(-1,0) -- ++(-1,-1) -- ++(0,-1);
%Mittelpunkt is (7,0)
\draw (8,1) -- ++ (1,1);
\draw              (7,1) -- ++ (0,1); \draw (7,2) -- ++(1,1);  \draw (7,2) -- ++(-1,0); 
\draw (6,0) -- ++(-1,0); 
\draw             (6,-1) -- ++(-1,-1); \draw (5,-2) -- ++(0,-1); \draw (5,-2) -- ++ (-1,0); 
\draw  (7,-1) -- ++(0,-1); 
\draw              (8,0) -- ++(1,0); \draw (9,0) -- ++(1,1); \draw (9,0) -- ++(0,-1);

\end{tikzpicture}
\caption{The skeleton and tropicalization of a Tate curve.}
\label{figure tate curve}
\end{figure}
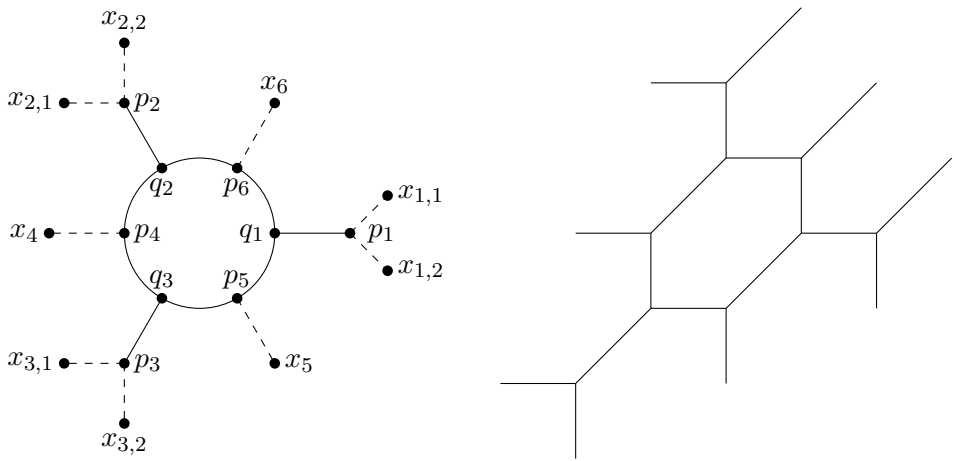

\begin{Ex} \label{counterexample}
In the same example, we can also see that Theorem \ref{lifting theorem}
does not hold if we do not require $B$ to be supported on two-valent points. 
Let $D = p_1$. Then the unique break divisor that is linearly equivalent to $D$ is $B = q_1$. 
However we cannot find $x$ and $y$ such that $\tau(x) = p_1$ and $\tau(y) = q_1$ such 
that $x-y$ is principal, since no difference of two distinct points is principal on an elliptic curve. 
\end{Ex}

%--------------------------------------------------------------------------------------------------------------------------------------

\section{Fully faithful and smooth tropicalizations}

\subsection{Describing tropicalizations using extended skeleta}

Let $X$ be smooth projective curve of genus $g>0$. 
Let $V$ be a minimal semistable vertex set of $X$ with associated finite skeleton $\Gamma$ and retraction $\tau$.

\begin{defn}
Let $\Sigma$ be a completed skeleton of $X$  with retraction $\tau_{\Sigma}$, $f \in K(X)^*$ 
and write $\divisor(f) = \sum \pm x_i$. 
Then $f$ is said to be \emph{faithful} with respect to $\Sigma$ 
if we have $\tau_{\Sigma}(x_i) \neq \tau_{\Sigma}(x_j)$ for all $i \neq j$.  
\end{defn}

Note that this implies that $f$ has only simple poles and zeros. 

\begin{Const} \label{higher tropicalization}
Let $\varphi \colon X \to Y$ be a closed embedding of $X$ into a toric variety $Y$ that meets the dense torus. 
Let $\Sigma(\varphi)$ be the completed skeleton associated to $\varphi$.  
Let $f \in K(X)^*$ be faithful with respect to $\Sigma(\varphi)$.
Consider the induced closed embedding $\varphi' = (\varphi, f) \colon X \to Y \times \Pbb^1$. 

We obtain the associated skeleton $\Sigma(\varphi')$ for $\varphi'$ by 
adding infinite rays $[x_i, \tau_{\varphi}(x_i)]$ for all $x_i \in \supp(\divisor(f))$.  
We denote by $\tau_{\varphi'}$ the associated retraction. 

We have the following diagram
\begin{align*}
\begin{xy}
\xymatrix{
\Xan \ar[rr]^{\tau_{\varphi'}} \ar[rrd]_{\tau_{\varphi}} &&\Sigma(\varphi') \ar[rr]^{\varphi'_{\trop}} \ar[d] && 
\Trop_{\varphi'}(X) \ar[d] \ar[r] & \Trop(Y) \times \Trop(\Pbb^1) \ar[d]^{\pi_1} \\
&& \Sigma(\varphi) \ar[rr]^{\varphi_{\trop}}  && \Trop_{\varphi}(X) \ar[r] &\Trop(Y).
}
\end{xy}
\end{align*}
The map on the left contracts the edges $[x_i, \tau_{\varphi}(x_i)]$ to $\tau_{\varphi}(x_i)$.
The map $\pi_1$ on the right is forgetting the last coordinate. 

Thus we obtain $\Trop_{\varphi'}(X)$ from $\Trop_\varphi(X)$ in two steps:
\begin{enumerate}
\item 
Take the graph of $\log \vert f \vert$ restricted to $\Sigma(\varphi)$. 
\item
Add the images of the edges $e_i = [x_i, \tau(x_i)]$. 
These are infinite rays from $(\varphi_{\trop}(x_i), \log \vert f(x_i) \vert)$ to  $(\varphi_{\trop}(x_i), \pm \infty)$
where the sign of $\infty$ is the opposite of the sign of $x_i$ in $\divisor(f)$. 
\end{enumerate}
\end{Const}

\begin{lem} \label{lem new edges}
In the situation of Construction \ref{higher tropicalization}, 
every edge $e$ in $\Sigma(\varphi')$ that is not an edge of $\Sigma(\varphi)$ is infinite and 
satisfies $m(e) = 1$.
\end{lem}

\begin{proof}
The edge $e$ has to be infinite since we only added infinite rays to $\Sigma(\varphi)$ 
in Construction \ref{higher tropicalization}. 
Since $f$ has only simple poles and zeros, the slope of $\log \vert f \vert$ along $e$ is equal to one, 
thus the corresponding expansion factor equals one. 
\end{proof}

%--------------------------------------------------------------------------------------------------------------------------------------

\subsection{Fully faithful tropicalization} \label{section ff trop}

Throughout this section, $X$ is a Mumford curve
and $\varphi \colon X \to Y$ a closed embedding of $X$ into a toric variety that meets the dense torus. 
In this section, we prove Theorem \ref{thmB} from the introduction, showing that $\varphi$ has 
a refinement that is fully faithful. 

We fix a minimal semistable vertex set $V$ and denote by $\Gamma$ the corresponding finite skeleton of $X$ 
with retraction $\tau$. 
For our completed skeleton $\Sigma(\varphi)$ associated to $\varphi$
we denote the retraction by $\tau_{\varphi}$ and the finite part by $\Sigma(\varphi)_{\fin}$.

We will now construct for an edge $e$ a function $f_e \in K(X)^*$ such that the slope
of $\log \vert f_e \vert$ is equal to $1$ along $e$ and such that $f_e$ is faithful
with respect to $\Sigma(\varphi)$.

\begin{Const} \label{construction finite edge}
Let $e$ be a finite edge of $\Sigma(\varphi)$ that is not in $\Gamma$.
We label the endpoints $v$ and $w$ in such a way that $w$ and $\Gamma$ 
lie in different connected components of $\Sigma(\varphi) \setminus v$
(see Figure \ref{figure finite edges}).
Let $v', w' \in X(K)$ be such that $\tau_{\varphi}(v') = v$ and $\tau_{\varphi}(w') = w$.
We fix edges $e_1,\dots,e_g$ that form the complement of a spanning tree of $\Sigma(\varphi)$ and
pillar points $p_{ij}^e$ in $e_i$. 
Applying Corollary \ref{lifting corollary} to $\Sigma(\varphi)_{\fin}$ 
and $D' = v' - w'$ we obtain $f_e \in K(X)^*$ such that
$\divisor(f_e) = v' - w' + \sum \pm x_{ij}^e$. 
By construction $f_e$ is faithful with respect to $\Sigma(\varphi)$ and 
the slope of $\log \vert f_e \vert$ along $e$ is $1$. 
Replacing $f_e$ by $a^{-1} \cdot f_e$ where $a \in K$ such that $\vert f_e(v) \vert = \vert a \vert $ 
we may assume $\log \vert f_e (v) \vert = 0$. 

\end{Const}

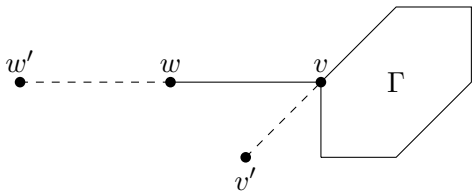
\begin{figure} 
\begin{tikzpicture}
\draw  (0,0) -- ++(4,0) -- (5,1) -- (5,2) -- (1,2) -- (0,1) -- (0,0);
\draw (-2, 1) -- (0,1);
\node[above] at (0,1) {$v$}; \fill (0,1) circle (2pt);
\node[above] at (-2,1) {$w$}; \fill (-2,1) circle (2pt);
\node[above] at (-1,1) {$e$}; \fill (-2,1) circle (2pt);
\draw[dashed] (-4,1) -- (-2,1); 
\draw[dashed] (-1,0) -- (0,1);
\node[above] at (-4,1) {$w'$}; \fill (-4,1) circle (2pt);
\node[below] at (-1,0) {$v'$}; \fill (-1,0) circle (2pt);
\node[above] at (1.5,0) {$e_1$};
\node[below] at (.2,0) {$p^e_{1,1}$}; \fill (.2,0) circle (2pt);
\node[below] at (.9,0) {$p^e_{1,2}$};  \fill (.9,0) circle (2pt);
\node[below] at (2.2,0) {$p^e_{1,3}$};  \fill (2.2,0) circle (2pt);
\node[below] at (2.9,0) {$p^e_{1,4}$};  \fill (2.9,0) circle (2pt);
\node at (2.5,1) {$\Gamma$};
\end{tikzpicture}
\caption{Situation in Construction \ref{construction finite edge}. 
The dashed lines are infinite edges and solid lines are finite edges.}
\label{figure finite edges}
\end{figure}

\begin{Const} \label{construction infinite edge}
Let $e$ be an infinite edge of $\Sigma(\varphi)$
with finite vertex $v$ and infinite vertex $w'$. 
Let $v'$ be a point in $X(K)$ such that $\tau_{\varphi}(v') = v$
(see Figure \ref{figure infinite edge}). 
We fix edges $e_1,\dots,e_g$ that form the complement of a spanning tree of $\Sigma(\varphi)_{\fin}$
and pillar points $p_{ij}^e$ in $e_i$.  
Applying Corollary \ref{lifting corollary} to $\Sigma(\varphi)_{\fin}$ and $D = v' - w'$ we obtain $f_e \in K(X)^*$
that is faithful with respect to $\Sigma(\varphi)$ and such that $\log \vert f_e \vert$ 
has slope $1$ along $e$. 
We again normalize such that $\log \vert f_e(v) \vert = 0$. 
\end{Const}

\begin{figure}
\begin{tikzpicture}
\draw  (0,0) -- ++(4,0) -- (5,1) -- (5,2) -- (1,2) -- (0,1) -- (0,0);
\node[above] at (0,1) {$v$}; \fill (0,1) circle (2pt);
\draw[dashed] (-4,1) -- (0,1); 
\draw[dashed] (-1,0) -- (0,1);
\node[above] at (-2,1) {$e$};
\node[above] at (-4,1) {$w'$}; \fill (-4,1) circle (2pt);
\node[below] at (-1,0) {$v'$}; \fill (-1,0) circle (2pt);
\node[above] at (1.5,0) {$e_1$};
\node[below] at (.2,0) {$p^e_{1,1}$}; \fill (.2,0) circle (2pt);
\node[below] at (.9,0) {$p^e_{1,2}$};  \fill (.9,0) circle (2pt);
\node[below] at (2.2,0) {$p^e_{1,3}$};  \fill (2.2,0) circle (2pt);
\node[below] at (2.9,0) {$p^e_{1,4}$};  \fill (2.9,0) circle (2pt);
\node at (2.5,1) {$\Gamma$};
\end{tikzpicture}
\caption{Situation in Construction \ref{construction infinite edge}.
The dashed lines are infinite edges and solid lines are finite edges.}
\label{figure infinite edge}
\end{figure}
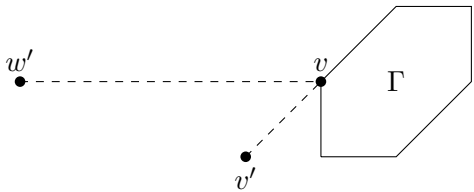

\begin{satz} \label{prop edge}
Let $X$ be a Mumford curve. 
Let $\varphi \colon X \to Y$ be a closed embedding of $X$ into a toric variety that meets the dense torus.
Then there exists a refinement $\varphi' \colon X \to Y'$ for $\varphi$ that is fully faithful. 
\end{satz}

\begin{proof}
Recall that we fixed a finite skeleton of $\Gamma$ of $X$. 
By \cite[Theorem 1.1]{BPR} we may assume, after possibly replacing $\varphi$ by a refinement, 
that the map $\varphi_{\trop}\vert_\Gamma$ is an isometry onto its image. 

Let $E$ be the set of edges of $\Sigma(\varphi)$ that are not in $\Gamma$. 
The strategy of proof will be as follows: For each edge $e \in E$, 
we apply Construction \ref{construction finite edge} (if $e$ is a finite edge) or 
Construction \ref{construction infinite edge} (if $e$ is an infinite edge). 
We make sure that the pillar points we choose to apply these constructions do not get 
in the way of each other (condition $iii)$ below) and do not interfere with $e$ after tropicalization
(condition $ii)$ below). 
This yield a rational function $f_e$ for each $e \in E$. 
We then check that the corresponding embedding 
$(\varphi, (f_e)_{e \in E}) \colon X \to Y \times (\mathbb{P}^1)^{\vert E \vert}$ 
is fully faithful. 

For each $i = 1,\dots,g$, $j=1,\dots,4$ and $e \in E$ we pick $p_{ij}^e \in \Gamma$ such that
\begin{enumerate}
\item
for all $e \in E$ there are edges $e^e_i$, $i = 1, \dots, g$, 
that form the complement of a spanning tree of $\Sigma(\varphi)$ and  
$p_{i,1}^e,\dots,p_{i,4}^e$ are pillar points in $e^e_{i}$;
\item
$\varphi_{\trop}([p_{i,1}^e, p_{i,4}^e]) \cap \varphi_{\trop}(e) = \emptyset$ for all $i = 1,\dots, g$;
\item
$[p_{i,1}^e, p_{i,4}^e] \cap [p_{i',1}^{e'}, p_{i',4}^{e'}] = \emptyset$ for $(e,i) \neq (e', i')$. 
\end{enumerate}
Note that a choice of $p_{ij}^e$ that satisfies ii) is possible since $\varphi_{\trop}(e)$ is a line segment, 
thus cannot cover a full cycle of $\Gamma$. 

Now for all finite (resp. infinite) edges $e \in E$ 
we apply Construction \ref{construction finite edge} (resp. Construction \ref{construction infinite edge}) 
and obtain functions $f_e \in K(X)^*$.

We consider the closed embedding 
\begin{align*}
\varphi' := (\varphi, (f_e)_{e \in E}) \colon X \to Y \times (\Pbb^1)^{\vert E \vert}.
\end{align*}
Following Construction \ref{higher tropicalization}, the completed skeleton 
$\Sigma(\varphi')$ associated to $\varphi'$ is obtained from $\Sigma(\varphi)$ 
by attaching an infinite edge $e_{ij}^e$ at each 
$p_{ij}^e$ and by attaching for each $e \in E$ infinite edges to its finite endpoints. 
If $e = [v, w]$, we denote these edges by $e_v^e$ and $e_w^e$ respectively. 
We claim that the map 
\begin{align*}
\varphi'_{\trop} \colon \Sigma(\varphi') \to \Trop(Y) \times \Trop(\Pbb^1)^E
\end{align*}
is injective. 
We denote $F_e :=  \log \vert f_e \vert \big \vert_{\Sigma(\varphi')}$. 
By construction, $\varphi'_{\trop}$ is injective when restricted to an edge, 
since $\varphi_{\trop}|_\Gamma$ is injective and 
$F_e$ is injective when restricted to $e$ and $e_{ij}^e$ for $e \in E$. 

To show global injectivity, let us set up some notation. 
Recall that for each edge $e \in E$ we denote by $v_e$ the endpoint of $e$ such that $\Gamma$ and 
$\mathring e$ lie
in different connected components of $\Gamma \setminus v$ and by $w_e$ the other endpoint. 
Furthermore, $f_e$ was normalized in such a way that $F_e (v_e) = 0$. 
Recall that $\Gamma$ is a deformation retract of $\Sigma(\varphi')$. 
Thus, we may define a partial order on $E$ by declaring $e \leq e'$ if ``$e$ is closer to $\Gamma$ 
then $e'$'', meaning that $\mathring e$ and $\mathring e'$ lie in the same connected component 
of $\Sigma(\varphi') \setminus v_e$. 

The idea of the proof of injectivity is that for a point $z \in \Sigma(\varphi') \setminus \Gamma$ 
we can detect in which edge $e \in E$ it is contained simply by looking at the 
set of functions $F_e$ satisfying $F_e(z) \neq 0$. 
We then do a case by case analysis of the situation.

Now assume $\varphi'_{\trop}(z_1) = \varphi'_{\trop}(z_2)$ for $z_1,z_2 \in \Sigma(\varphi')$. 
This means that $\varphi_{\trop}(z_1) = \varphi_{\trop}(z_2)$ and 
$F_e(z_1)  = F_e(z_2)$ for all $e \in E$. 

Note that we may assume $z_1 \notin \Gamma$, since if both $z_1$ and $z_2$ are in $\Gamma$, 
then we are done since $\varphi_{\trop}$ is already injective on $\Gamma$. 
Denote 
\begin{align*}
E' := \{ e \in E \mid F_e(z_1)  \neq 0 \} = \{ e \in E \mid F_e(z_2) \neq 0 \}.
\end{align*} 
Since $F_e(v_e) = 0$ and 
$\divisor(F_e) = v_e - w_e + \sum_{i=1}^g (p_{i,1}^e - p_{i,2}^e - p_{i,3}^e +  p_{i,4}^e)$  we have 
$F_e(v_e) = 0$, $F_e(w_e) > 0$, $F_e(p_{i,1}^e) = F_e(p_{i,4}^e) = 0$ and 
$F_e$ is constant on the connected components 
of $\Sigma(\varphi') \setminus (e \cup [p_{i,1}^e, p_{i,4}^e])$ (see Figure \ref{figure graph}). 
Thus
\begin{align} \label{eq:fe}
\supp(F_e) = \bigcup_{e' \geq e} e' \cup \bigcup_{i=1}^g [p^e_{i1}, p^e_{i4}].  
\end{align}
We deduce that $E'$ is closed under $\leq$ and  
non-empty since $z_1 \notin \Gamma$. 

If $\vert E' \vert = 1$, say $E' = \{e \}$, then 
\begin{align} \label{eq zi contained}
z_1 \in e \cup \bigcup_{ij} e_{ij}^e \text{ and }  z_2 \in e \cup \bigcup_{ij} e_{ij}^e \cup 
\bigcup_{i} [p_{i,1}^e, p_{i,4}^e]. 
\end{align}
In the case $z_1 \in e$, we have $\varphi_{\trop}(z_2) = \varphi_{\trop}(z_1) \in \varphi_{\trop}(e)$ 
which forces $z_2 \in e$ by $ii)$ above and (\ref{eq zi contained}). 
Since $F_e  |_{e}$ is injective, it follows that $z_1 = z_2$. 

In the case $z_1 \in e_{ij}^e$, we have $\varphi_{\trop}(z_2) = \varphi_{\trop}(z_1) = p_{ij}^e$, 
thus $z_2 \in e_{ij}^e$ and because $F_e|_{e_{ij}^e}$ is injective 
we have $z_1 = z_2$.

If $\vert E' \vert > 1$, then there exists $e \in E$ such that 
$E' = \{ e' \in E \mid e' \leq e\}$ by iii) above and (\ref{eq:fe}).
For the same reason $\vert E' \vert > 1$ implies $z_1, z_2 \in e$ and consequently $z_1 = z_2$.  
Thus $\varphi'_{\trop}|_{\Sigma(\varphi')}$ 
is injective. 

The stretching factor for all edges of $\Gamma$ is $1$ since $\varphi |_{e}$ is an isometry onto its image. 
For all $e \in E$, the stretching factors are equal to $1$ since the slope of $f_e$ along $e$ is $1$. 
For all $e^e_{ij}$ the stretching factor is equal to $1$ by Lemma \ref{lem new edges}. 
Since $\varphi'_{\trop}|_{\Sigma(\varphi')}$ in injective this means all weights are equal to $1$. 
Thus $\varphi'_{\trop}$ represents $\Sigma(\varphi')$ faithfully. 
\end{proof}

\begin{figure}
\begin{tikzpicture}
\draw (-1,3) --  ++(1,0) -- ++(1,1) -- ++(2,0) -- ++(1,-1) -- ++(1,0);
\draw[dashed] (0,3) -- ++(0,-1); 
\draw[dashed] (1,4) -- ++(0,1); 
\draw[dashed] (3,4) -- ++(0,1);
\draw[dashed] (4,3) -- ++(0,-1);
\node [above] at (0,3) {$p_{i,1}^e$}; \fill (0,3) circle (2pt);
\node [below] at (1,4) {$p_{i,2}^e$}; \fill (1,4) circle (2pt);
\node [below] at (3,4) {$p_{i,3}^e$}; \fill (3,4) circle (2pt);
\node [above] at (4,3) {$p_{i,4}^e$}; \fill (4,3) circle (2pt);
\node[right] at (0,2.5) {$e_{i,1}^e$};
\node[right] at (1,4.5) {$e_{i,2}^e$};
\node[right] at (3,4.5) {$e_{i,3}^e$};
\node[right] at (4,2.5) {$e_{i,4}^e$};
\end{tikzpicture}
\caption{The graph of $\log \vert f_e \vert$ on $e_{i}^e$ and the 
adjacent edges $e_{ij}^e$ in $\Sigma(\varphi')$.}
\label{figure graph}
\end{figure}
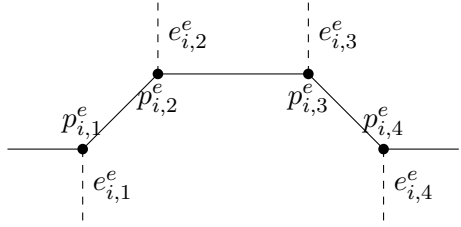

\begin{kor} \label{cor section}
Let $\varphi \colon X \to Y$ be a closed embedding of $X$ into a toric variety $Y$ that meets the dense torus. 
Then there exists a refinement $\varphi'$ of $\varphi$ and a 
section $\psi_{\varphi'} \colon \Trop_{\varphi'}(X) \to \Xan$ for $\varphi'_{\trop}$.
\end{kor}
\begin{proof}
By Theorem \ref{prop edge} we can choose $\varphi'$ such that $\varphi'_{\trop}$ is fully faithful. 
Thus $\varphi'_{\trop}\vert_{\Sigma(\varphi')}$ is a homeomorphism and we 
define $\psi_{\varphi'}$ as the composition of the inclusion of $\Sigma(\varphi')$
into $\Xan$ with $(\varphi'_{\trop}\vert_{\Sigma(\varphi')})^{-1}$. 
\end{proof}

%--------------------------------------------------------------------------------------------------------------------------------------

\subsection{Smooth tropicalization} \label{section smooth trop}

Throughout this section, we will work in the following situation: 
$X$ is a Mumford curve over $K$ and $\varphi \colon X \to Y$ a closed embedding that meets the dense torus
such that $\varphi_{\trop}$ is fully faithful. 
We denote by $\Sigma(\varphi)$ the associated complete skeleton and by $\tau_\varphi$ the retraction.

\begin{lem} \label{less singular points}
Let $f \in K(X)^*$ that is faithful with respect to $\Sigma(\varphi)$.
Then $\varphi' = (\varphi, f) \colon X \to Y \times \Pbb^1$ is fully faithful. 
Further, all vertices in $\Sigma(\varphi')$ that map to singular vertices in $\Trop_{\varphi'}(X)$
are contained in $\Sigma(\varphi)$ and map to singular vertices in $\Trop_{\varphi}(X)$.
\end{lem}
\begin{proof}
All edges of $\Sigma(\varphi')$ that are not edges of $\Sigma(\varphi)$ 
have expansion factor equal to $1$ by Lemma \ref{lem new edges}. 
Since $f$ is faithful all these edges 
have different images under $\tau_{\varphi}$. 
Since $\varphi_{\trop}$ is fully faithful, they have different images under $\varphi'_{\trop}$. 
Consequently $\varphi'_{\trop}|_{\Sigma(\varphi')}$ is injective. 
Thus $\varphi'_{\trop}$ is fully faithful. 

Let $v$ be a vertex of $\Sigma(\varphi')$.
Then $v$ is a vertex in $\Sigma(\varphi)$ (after potential subdivision) or infinite.
Since $\varphi'_{\trop}$ is fully faithfully, 
the infinite vertices of $\Trop_{\varphi'}(X)$ have only one adjacent vertex 
and are thus smooth.  
Thus we have to show that if $\varphi_{\trop}(v)$ is a smooth finite vertex of $\Trop_{\varphi}(X)$, 
then $\varphi'_{\trop}(v)$ is a smooth vertex of $\Trop_{\varphi'}(X)$. 

Let $e_0,\dots,e_n$ be the adjacent edges of $v$ and write $w_i := w_{v, e_i}$ for
the primitive integral vector pointing from $\varphi_{\trop}(v)$ into $\varphi_{\trop}(e_i)$. 
We denote $F =  \log \vert f \vert \big \vert_{\Sigma(\varphi')}$ and $L(F) = F - F(v)$.  

If $v$ is not in $\divisor(F)$, then $F$ is locally around $\varphi_{\trop}(v)$ 
the restriction  to $\Trop_\varphi(X)$ of an affine function on $N_\R$.
The vertex $v$ still has $n+1$ adjacent edges $e'_0,\dots,e'_n$ in $\Sigma(\varphi')$ and the 
primitive vectors are $w_0' = (w_0,L(F)(e_0)),\dots,w'_n = (w_n,L(F)(e_n))$. 
Since $F$ has integer slopes and is the restriction of an affine function and the $w_i$ span
a saturated lattice of rank $n$, so do the $w'_i$, 
which shows that $\varphi_{\trop}(v)$ is a smooth vertex of $\Trop_{\varphi'}(X)$. 

If $v \in \divisor(F)$, since $f$ is faithful with respect to $\Sigma(\varphi)$, 
$v$ has $n+2$ adjacent edges $e'_0,\dots,e'_n,e'_{n+1}$ 
 in $\Sigma(\varphi')$ and the primitive vectors are
\begin{align*}
w_0' = (w_0,L(F)(e_0)),\dots,w'_n = (w_n,L(F)(e_n)), w'_{n+1} = (0, \pm 1). 
\end{align*}
Since the $w_i$ span a saturated lattice of rank $n$, 
the $w'_i$ span a saturated lattice of rank $n+1$, 
which shows that $\varphi'_{\trop}(v)$ is a smooth vertex of $\Trop_{\varphi'}(X)$. 
\end{proof}

For a vertex $v$ of $\Sigma(\varphi)$ and two adjacent edges $e_0$ and $e_1$, we 
now construct a function $f_{e_1}$ in $K(X)^*$ that we will use to
construct a tropicalization that is smooth at $v$. 
This may be viewed as generalization to any ambient dimension and any vertex 
of the constructions done for special vertices and ambient dimension $2$ by 
Cueto and Markwig \cite[Section 3]{CuetoMarkwig} 

\begin{Const} \label{new construction vertex}
Let $v$ be a vertex of $\Sigma(\varphi)$ and let $e_0$ and $e_1$ be adjacent edges. 
Let $F_{e_1}$ be a piecewise linear function such that $F_{e_1}(v) = 0$, $d_{e_1}F_{e_1}(v) = 1$, 
$d_{e_0}F_{e_1}(v) = -1$, $d_{e}F_{e_1} = 0$ for all other adjacent edges, 
$\supp(F_{e_1}) \subset e_1 \cup e_0$, and such that 
$\divisor(F_{e_1}) = \sum \pm p_{i}$ for distinct points $p_i$ (see Figure \ref{figure resolution}). 
For each $i$ fix $x_{i} \in X(K)$ such that $\tau_* x_{i} = p_i$ and write $D = \sum \pm x_{i}$.
Fixing pillar points $p_{jk}$ outside of $\supp(F)$ and 
applying Corollary \ref{lifting corollary} we obtain a function $f_{e_1} \in K(X)^*$
that is faithful with respect to $\Sigma(\varphi)$ and 
such that the outgoing slope at $v$ of $\log \vert f_{e_1} \vert$ equals $1$ along $e_1$ and $-1$ along $e_0$. 
\end{Const}

\begin{figure}
\begin{tikzpicture}
\draw  (0,0) -- (1,1) -- (3,1) -- (4,0) -- (5,0);
\draw  (0,0) -- (-1,-1) -- (-3,-1) -- (-4,0) -- (-5,0);
\draw[dashed] (-4,0) -- (4,0);

\draw (-5,-1.3) -- (5,-1.3); 
\node[below] at (-2.5,-1.3) {$e_0$}; 
\node[below] at (2.5,-1.3) {$e_1$}; 
\node [below] at (0,-1.3) {$v$}; \fill (0,-1.3) circle (2pt);
\end{tikzpicture}

\caption{The graph of $F_{e_1}$ along the edges $e_0$ and $e_1$ in 
Construction \ref{new construction vertex}, with the dashed line
being the zero level.}
\label{figure resolution}
\end{figure}
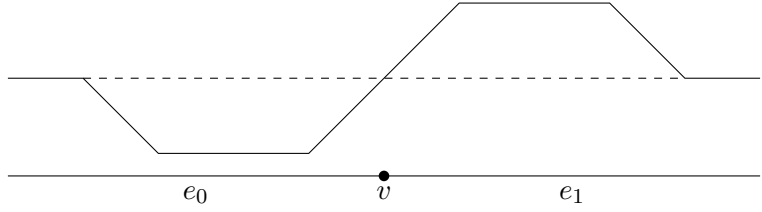

\begin{satz} \label{main theorem}
Let $\varphi \colon X \to Y$ be an embedding of $X$ into a toric variety $Y$ that meets the dense torus.
Then there exists a refinement $\varphi' \colon X \to Y'$ for a toric variety $Y'$ 
such that $\Trop_{\varphi'}(X)$ is a smooth tropical curve. 
\end{satz}

\begin{proof}
By Theorem \ref{prop edge}, after replacing $\varphi$ by a refinement, 
we may assume that $\varphi_{\trop}$ is fully faithful.

Let $v$ be a vertex of $\Sigma(\varphi)$ such that $\varphi_{\trop}(v)$ is a singular vertex of $\Trop_{\varphi}(X)$. 
Let $e_0,\dots,e_n$ be the adjacent edges. 
For $k=1,\dots,n$ we pick functions $F_{e_i} \colon \Sigma(\varphi) \to \R$ 
as in Construction \ref{new construction vertex}. 
For each $k = 1,\dots,n$ and $i = 1,\dots,g$  
we pick pillar points $p^k_{i,1}, p^k_{i,2}, p^k_{i,3},p^k_{i,4}$ in such a 
way that 
\begin{align*}
[p^k_{i,1}, p^k_{i,4}] &\cap \supp (F_k) = \emptyset \text{ for all } i,k \text{ and } \\
[p^{k}_{i,1}, p^k_{i,4} ] &\cap [p^{k'}_{i',1}, p^{k'}_{i',4}] = \emptyset \text{ for }(k,i) \neq (k',i'). 
\end{align*}
Applying now Construction \ref{new construction vertex}, we obtain functions $f_{e_i} \in K(X)$. 
We consider the closed embedding $\varphi' := (\varphi, (f_{e_k})_{k=1,\dots,n}) \colon X \to (\Pbb^{1})^n$ 
and its tropicalization
\begin{align*}
\varphi'_{\trop} \colon \Xan \to \Trop(Y) \times \Trop(\Pbb^1)^n.
\end{align*}
Applying Lemma \ref{lem new edges} $n$ times, we see that $\varphi'_{\trop}$
is fully faithful. 
By construction $v$ still has $n+1$ adjacent edges $e'_0,\dots,e'_n$ in $\Sigma(\varphi')$ and 
$\log \vert f_{e_i} \vert$ has 
slope $1$ along $e'_i$, slope $-1$ along $e'_0$, and is constant on the other edges. 
This means that projecting a neighborhood of $\varphi'_{\trop}(v)$ in 
$\Trop_{\varphi'}(X) \subset \Trop(Y) \times \Trop(\Pbb^1)^n$ 
to the second factor, the image is isomorphic to the one-dimensional fan in $\R^n$ 
whose rays are spanned by the coordinate vectors $x_1,\dots,x_n$ and 
their negative sum $x_0 = - \sum_{i=1}^n x_i$.  
Further the primitive vector $w_{v,e'_i}$ is mapped to $x_i$. 
Thus the $w_{v,e'_i}$ span a saturated lattice of rank $n$, which means that $v$ is smooth in $\Trop_{\varphi'}(X)$.

Since $v$ is singular in $\Trop_{\varphi}(X)$ but not in $\Trop_{\varphi'}(X)$, 
by inductively applying Lemma \ref{less singular points}, 
we see that $\Trop_{\varphi'}(X)$ 
has fewer singular points than $\Trop_{\varphi}(X)$. 

Thus inductively we can construct $\varphi'$ such that $\Trop_{\varphi'}(X)$ is smooth. 
\end{proof}

%\begin{kor} \label{main theorem very affine}
%Let $X$ be a smooth curve such that all type II points in $\Xan$ have genus $0$. 
%Then there exists a very affine open subset $U$ of $X$ and a closed 
%embedding $\varphi \colon U \to T$ for a torus $T$ such that $\Trop_\varphi(U)$ 
%is a smooth tropical curve in $\R^n$. 
%\end{kor}
%
%\begin{proof}
%Let $\overline{X}$ be the unique smooth projective curve that 
%has the same functions field as $X$. 
%Then $X$ is isomorphic to a dense open subset of $\overline{X}$. 
%We thus consider $X$ as an open subset of $\overline{X}$. 
%Hence all type II points of $\overline{X}$ have genus $0$. 
%Thus replacing $X$ by $\overline{X}$ 
%we may assume that $X$ is a Mumford curve. 
%Then by Theorem \ref{main theorem} there exists a closed
%embedding $\varphi \colon X \to Y$ for a toric variety $Y$ such 
%that $\Trop_{\varphi}(X)$ is smooth. 
%Denote by $T$ the dense torus of $Y$. 
%Then we take $U := \varphi^{-1}(T)$ and the closed embedding $\varphi|_{U} \colon U \to T$.  
%\end{proof}

%------------------------------------------------------------------------------------------------------------------------------

\section{Only Mumford curves admit smooth tropicalizations} \label{section smooth mumford}

Let $X$ be a smooth projective curve. 
In this section we show that the existence of a closed embedding $\varphi \colon X \to Y$ such 
that $\Trop_{\varphi}(X)$ is smooth already implies that $X$ is a Mumford curve. 
Since we will not change the embedding in this section, we 
will identify $X$ with its image and simply treat $X$ as a closed subcurve of $Y$. 
We denote the completed skeleton associated to the inclusion of $X$ into $Y$ by $\Sigma$. 

We denote by $K^\circ$ the valuation ring of $K$ and by $\Ktilde$ its residue field. 
Further we denote by $T$ the dense torus of $Y$, by $N$ its cocharacter lattice and 
$N_\Lambda = N \otimes \Lambda \subset N_\R$. 

We will use the notion of affinoid domains in $\Xan$ and their formal models. 
For an introduction to these notions we refer the reader to \cite[Section 3]{BPR}.

\begin{defn}
Let $w \in N_\Lambda \cap \Trop(X)$.
Then $X^w := \trop^{-1}(w)$ is an affinoid domain in $\Xan$. 
The point $w$ determines a formal model $\Xcal^w$ for $X^w$.  

The \emph{initial degeneration} is the special fiber $\inn_w(X) := \XS^w_s :=  \XS^w \otimes_{K^\circ} \Ktilde$.  
\end{defn}

\begin{bem} \label{lem initial smooth}
Assume that $\inn_w(X)$ is reduced. 
By \cite[Proposition 3.13]{BPR} we have that $\XS^w$ is 
the canonical model of $X^w$. 
Then we have a canonical \emph{reduction map} 
$\red \colon X^w \to \inn_w(X)$ \cite[Section 2.4]{BerkovichSpectral}. 
Let $C$ be an irreducible component of $\inn_w(X)$ with generic point $\eta$. 
Then there is a unique point $x_w \in X^w$ such that $\red(x_w) = \eta$ and that point satisfies 
that $C_{x_w}$ is birational to $C$ \cite[Proposition 2.4.4]{BerkovichSpectral}. 
If $z$ is a smooth closed point of $\inn_w(X)$ 
then $\red^{-1}(z)$ is isomorphic to an open disc \cite[Proposition 2.2]{BL}.

In particular, if $\inn_w(X)$ is smooth and rational, then all type II points in $X^w$ 
have genus $0$. 
\end{bem}

We will use the following Proposition. 
Since we will apply it in the case of a trivially valued field, 
we allow the absolute value of the field to be trivial. 

\begin{prop} \label{prop translate contained}
Let $T$ be an algebraic torus over a non-archimedean field, 
whose absolute value may be trivial. 
Let $T'$ be a subtorus and let $U$ be a closed subvariety of $T$.
If $\Trop(U) \subset \Trop(T')$ then a translate of $U$ that has the same tropicalization as $U$ is contained in $T'$. 
\end{prop}
\begin{proof}
We consider the quotient torus $T / T'$. 
Denote by $\overline{U}$ the image of $U$ in the quotient torus $T / T'$. 
Then the tropicalization of $\overline{U}$ in $\Trop(T / T') = \Trop(T) / \Trop(T')$ 
is a point by construction, meaning that $U$ is contained in a translate $t \cdot T'$ of $T'$,
where all entries of $T$ have absolute value $1$. 
Thus $t^{-1} \cdot U$ is a translate of $U$ that is contained in $T'$ and has
the same tropicalization as $U$. 
\end{proof}

In the following, we view $\Ktilde$ as a non-archimedean field, carrying the trivial absolute value. 

\begin{satz} \label{duck theorem}
Let $T$ be an algebraic torus over $\Ktilde$. 
Let $U \subset T$ be a closed curve. 
If $\Trop(U)$ is smooth then $U$ is smooth and rational. 
\end{satz}
\begin{proof}
In the case where $\Trop(U)$ spans $\Trop(T)$, it follows from \cite[Proposition 4.2]{KatzPayne}
that the closure of $U$ in $\mathbb{P}^n$ is a one-dimensional linear space.
Thus $U$ is a smooth rational curve. 
We reduce to this case: 
Let $V$ be the vector subspace of $\Trop(T)$ that is spanned by $\Trop(U)$.
Since $V$ is a rational subspace, there exists a subtorus $T'$ of $T$ such that $\Trop(T') = V$. 
Now replacing $U$ by the translate from Proposition \ref{prop translate contained}
and applying Katz's and Payne's result to $U$ and $T'$ proves the theorem. 
\end{proof}

\begin{kor} \label{duck corollary}
If $\Trop(X)$ is smooth, then $\inn_w(X)$ is a smooth rational curve for 
all $w \in \Trop(X) \cap N_\Lambda$.
\end{kor}
\begin{proof}
Let $w \in \Trop(X) \cap N_{\Lambda}$. 
Then $\inn_w(X)$ is a closed subvariety of a torus $T_{\Ktilde}$ over $\Ktilde$. 
Denote by $\Trop(\inn_w(X))$ its tropicalization.  
Then the local cone at $w$ in $\Trop(X)$ equals $\Trop(\inn_w(X))$ by \cite[10.15]{Gubler2}.
Thus $\inn_w(X)$ is a smooth rational curve by Theorem \ref{duck theorem}. 
\end{proof}

\begin{satz} \label{main theorem II}
If $\Trop(X)$ is smooth, then $X$ is a Mumford curve. 
\end{satz}
\begin{proof}
Let $w \in \Trop(X) \cap N_\Lambda$. 
By Corollary \ref{duck corollary}, $\inn_w(X)$ is smooth and rational. 
Thus all type II points in $X^w$ have genus $0$ by Remark \ref{lem initial smooth}. 
Since all type two points map to $N_\Lambda$ under the tropicalization map, 
all type II points in $\Xan$ have genus zero which shows that $X$ is a Mumford curve by Remark \ref{bem Mumford}
\end{proof}

\begin{satz} \label{smooth implies ff}
If $\Trop(X)$ is smooth, then the tropicalization map is fully faithful.
\end{satz}
\begin{proof}
By Corollary \ref{duck corollary}, all initial degenerations are smooth and rational.
For all $w \in N_{\Lambda} \cap \Trop(X)$, by Remark \ref{lem initial smooth}, 
there is a unique point $x_w \in X^w$ 
that satisfies that $\red(x_w)$ is the generic point of $\inn_w(X)$. 
Furthermore, every point in $X^w \setminus \{x_w\}$ has 
a neighborhood isomorphic to an open disc, thus is not contained in $\Sigma$. 
We conclude that every point $w \in N_\Lambda \cap \Trop(X)$ 
has $x_w$ as its unique preimage under $\trop|_{\Sigma}$. 
Since $\trop|_{\Sigma}$ is continuous and linear on each edge, 
this implies that $\trop|_{\Sigma} \colon \Sigma \to \Trop(X)$ is bijective.
Since all weights are $1$, this shows that the tropicalization map is fully faithful. 
\end{proof}

Note that when $X$ comes by base change from a family of Riemann surfaces over the punctured disc, 
Theorems \ref{main theorem II} and \ref{smooth implies ff} are consequences of \cite[Corollary 2]{IKMZ}. 
The relation between Hodge and Betti numbers in tropical geometry is different than in complex geometry. 
The $(0,1)$-tropical Hodge number of $\Trop(X)$ is equal to the first Betti number of $\Trop(X)$. 
Using this and \cite[Corollary 2]{IKMZ} one finds that the first Betti number of $\Trop(X)$ is equal to $g$, 
which, since $\Trop(X)$ is smooth, implies that $\trop|_{\Sigma}$ is injective, hence bijective.

\bibliographystyle{alpha}
\def\cprime{$'$}

\end{document}